\documentclass[a4paper,12pt,reqno]{amsart}

\usepackage[T1]{fontenc}
\usepackage[utf8]{inputenc}
\usepackage{lmodern}
\usepackage[english]{babel}

\usepackage[%
	left=2.5cm,       
	right=2.5cm,      
	top=3.5cm,        
	bottom=3.5cm,     
	heightrounded,    
	bindingoffset=0mm 
]{geometry}

\usepackage{amsmath}
\usepackage{amssymb}
\usepackage{mathtools}
\usepackage{mathrsfs}

\usepackage{nicefrac}

\usepackage[normalem]{ulem}

\usepackage{microtype}

\usepackage[
colorlinks=true,
urlcolor=teal,
citecolor=magenta,
linkcolor=teal,
breaklinks=true]
{hyperref}

\usepackage[noabbrev,capitalize,nameinlink]{cleveref}

\usepackage{bookmark}

\usepackage[initials,msc-links,
]{amsrefs}

\usepackage[dvipsnames]{xcolor}
\usepackage{graphicx}
\usepackage{url}

\usepackage{enumitem}

\numberwithin{equation}{section}

\theoremstyle{plain}
\newtheorem{theorem}{Theorem}[section]
\newtheorem{lemma}[theorem]{Lemma}

\newtheorem{corollary}[theorem]{Corollary}

\theoremstyle{definition}

\newcommand{\di}{\,\mathrm{d}}
\newcommand{\dive}{\mathrm{div}}

\newcommand{\N}{\mathbb{N}}
\newcommand{\R}{\mathbb{R}}
\newcommand{\redb}{\mathscr{F}}

\newcommand{\weakstarto}{\xrightharpoonup{\star}}

\DeclareMathOperator{\lip}{Lip}
\DeclareMathOperator{\supp}{supp}
\DeclareMathOperator{\tanset}{Tan}

\DeclarePairedDelimiter{\set}{\{}{\}}

\usepackage[
final
]{showlabels}

\showlabels{bib}

\allowdisplaybreaks
\begin{document}

\title[On blow-ups of sets with finite fractional variation]{On blow-ups of sets with finite fractional variation}

\author[G.~Stefani]{Giorgio Stefani}
\address[G.~Stefani]{Università degli Studi di Padova, Dipartimento di
Matematica ``Tullio Levi-Civita'', Via Trieste 63, 35121 Padova (PD), Italy}
\email{giorgio.stefani@unipd.it}

\date{\today}

\dedicatory{To the memory of my beloved friend Diego.}

\keywords{Fractional gradient, fractional variation, tangent measure, blow-up, half-space}

\subjclass[2020]{Primary 49Q15. Secondary 26A33, 28A75, 49Q20}

\thanks{\textit{Acknowledgements}.
The author is member of the Istituto Nazionale di Alta Matematica (INdAM),
Gruppo Nazionale per l'Analisi Matematica, la Probabilità e le loro Applicazioni (GNAMPA), and has
received funding from INdAM under the INdAM--GNAMPA Project 2025 \textit{Metodi
variazionali per problemi dipendenti da operatori frazionari isotropi e
anisotropi} (grant agreement No.\ CUP\_E53\-240\-019\-500\-01), and from the European Union -- NextGenerationEU and the University of Padua under the 2023 STARS@UNIPD Starting Grant Project \textit{New Directions in Fractional
Calculus -- NewFrac} (grant agreement No.\ CUP\_C95\-F21\-009\-990\-001).
}

\begin{abstract}
Given $\alpha\in(0,1)$ and a set $E\subset\R^N$ with locally finite fractional $\alpha$-variation, we show that for $|D^\alpha\mathbf 1_E|$-a.e.\ $x$, every non-trivial tangent set of $E$ at $x$ with locally finite integer perimeter is a half-space oriented by the fractional inner unit normal of $E$ at $x$.
\end{abstract}

\maketitle

\section{Introduction}

\subsection{Setting}
Given $\alpha\in(0,1)$, the \emph{fractional $\alpha$-gradient} of $u\in\lip_c(\R^N)$ is defined as 
\begin{equation}
\label{eq:fracnabla}
\nabla^\alpha u(x)
=
c_{N,\alpha}
\int_{\R^N}\frac{(u(y)-u(x))\,(y-x)}{|y-x|^{N+\alpha+1}}\di y,
\quad
x\in\R^N,
\end{equation}
and the \emph{fractional $\alpha$-divergence} of $\varphi\in\lip_c(\R^N;\R^N)$ is analogously defined as
\begin{equation}
\label{eq:fracdiv}
\dive^\alpha\varphi(x)
=
c_{N,\alpha}
\int_{\R^N}\frac{(\varphi(y)-\varphi(x))\cdot(y-x)}{|y-x|^{N+\alpha+1}}\di y,
\quad
x\in\R^N,
\end{equation}
where $c_{N,\alpha}>0$ is a suitable normalization constant. 
The operators~\eqref{eq:fracnabla} and~\eqref{eq:fracdiv} satisfy the integration-by-parts formula 
\begin{equation}
\label{eq:ibp}
\int_{\R^N}u\,\dive^\alpha\varphi\di x
=
-\int_{\R^N}\varphi\cdot \nabla^\alpha u \di x.
\end{equation} 
For further details on the operators~\eqref{eq:fracnabla} and~\eqref{eq:fracdiv} and on the formula~\eqref{eq:ibp}, we refer the reader to~\cite{silhavy20}.
Building on the integration-by-parts formula~\eqref{eq:ibp}, in our previous works~\cites{brue-et-al22,comi-stefani19,comi-et-al22,comi-stefani22,comi-stefani23a,comi-stefani23b,comi-stefani24a,comi-stefani24b}, together with Giovanni E.~Comi and several collaborators, we developed a new theory of distributional fractional Sobolev and $BV$ spaces. 
The literature around the fractional gradient~\eqref{eq:fracnabla} is growing very rapidly in multiple directions, both theoretical and applicative. For a non-exhaustive list of contributions, we refer to~\cites{almi-et-al25,alicandro-et-al25,antil-et-al24,bellido-et-al20,bellido-et-al21,barrios-medina21,braides-et-al24,caponi-et-al25,carrero-et-al25,kreisbeck-schonberger22,liu-et-al24,liu-xiao22,lo-rodrigues23,schikorra-et-al17,schonberger24,shieh-spector15,shieh-spector18,spector19,spector20} and to the references therein.

\subsection{Fractional variation}

We continue the study of fractional $BV$ functions.
Given $p\in[1,\infty]$ and an open set $\Omega\subset\R^N$, we say that $u\in BV^{\alpha,p}_{\rm loc}(\Omega)$ if $u\in L^p(\R^N)$ and
\begin{equation*}
|D^\alpha u|(A)
=
\sup\set*{\int_{\R^N}u\,\dive^\alpha\varphi\di x : \varphi\in C^\infty_c(\R^N;\R^N),\ \|\varphi\|_{L^\infty}\le 1,\,\supp\varphi\subset A}<\infty
\end{equation*} 
for every open set $A\Subset\Omega$.
By Riesz's Representation Theorem, $u\in BV^{\alpha,p}_{\rm loc}(\Omega)$ if and only if $u\in L^p(\R^N)$ and there exists a locally finite (vector-valued) Radon measure $D^\alpha u\in\mathscr M_{\rm loc}(\Omega;\R^N)$ such that~\eqref{eq:ibp} holds with $D^\alpha u$ in place of the fractional $\alpha$-gradient for every $\varphi\in C^\infty_c(\R^N;\R^N)$ such that $\supp\varphi\subset\Omega$.
If $|D^\alpha u|(\Omega)<\infty$, then we write $u\in BV^{\alpha,p}(\Omega)$.
Note that the subscript `loc' in $BV^{\alpha,p}_{\rm loc}$ refers only to the local finiteness
of the \emph{fractional variation measure}, since $BV^{\alpha,p}_{\rm loc}$ functions are, by definition, in $L^p(\R^N)$.

\subsection{Blow-up Theorem}

If $\mathbf 1_E\in BV^{\alpha,\infty}_{\rm loc}(\R^N)$, then $|D^\alpha\mathbf 1_E|\in\mathscr M_{\rm loc}(\R^N)$ is the \emph{distributional fractional $\alpha$-perimeter measure} of $E\subset\R^N$.
In analogy with the classical theory~\cites{ambrosio-et-al00,maggi12}, 
the \emph{fractional reduced $\alpha$-boundary} $\redb^\alpha E$ of $E$ is the set of points 
$x\in\supp|D^\alpha\mathbf 1_E|$ at which the \emph{(measure-theoretic) inner unit fractional normal} exists, namely
\begin{equation*}
\nu^\alpha_E(x)
=
\lim_{r\to0^+}
\frac{D^\alpha\mathbf 1_E(B_r(x))}{|D^\alpha\mathbf 1_E|(B_r(x))}
\in\mathbb S^{N-1}.
\end{equation*}

The main feature of the fractional reduced boundary we are interested in here 
is its connection with the set of \emph{tangent} (or \emph{blow-up}) \emph{sets} of~$E$ at~$x$, 
that is, the set $\tanset(E,x)$ of all limit points of $\set*{\frac{E-x}{r}:r>0}$ in $L^1_{\rm loc}(\R^N)$ as $r\to0^+$.

\begin{theorem}
\label{res:fracblowup}
Let $\mathbf 1_E\in BV^{\alpha,\infty}_{\rm loc}(\R^N)$ and $x\in\redb^\alpha E$. Then, $\tanset(E,x)\ne\emptyset$ and any $F\in\tanset(E,x)$ is such that $\mathbf 1_F\in BV^{\alpha,\infty}_{\rm loc}(\R^N)$ with $\nu_F^\alpha=\nu_E^\alpha(x)$ for $|D^\alpha\mathbf 1_F|$-a.e.\ $y\in\redb^\alpha F$.
Moreover, assuming $x=0$ and $\nu^\alpha_E(0)=\mathrm e_N$ without loss of generality, $F=\R^{N-1}\times M$ for some measurable set $M\subset\R$ such that:
\begin{enumerate}[label=(\roman*),itemsep=1ex,topsep=1ex]

\item 
$\mathbf 1_M\in BV^{\alpha,\infty}_{\rm loc}(\R)$ with $\partial^\alpha\mathbf 1_M\ge0$;

\item
$|M|,|M^c|\in\set*{0,\infty}$;

\item
if $|M|=\infty$, then $\operatorname{ess}\sup M=\infty$;

\item
\label{item:finiteper}
if $M\notin\set{\emptyset,\R}$ and $P(M)<\infty$, then $M=(m,\infty)$ for some $m\in\R$.
\end{enumerate}  
\end{theorem}

The first part of \cref{res:fracblowup} follows from~\cite{comi-stefani19}*{Th.~5.8 and Prop.~5.9}, 
while the second part is taken from~\cite{comi-stefani24b}*{Th.~1.7}.
\Cref{res:fracblowup} can be regarded as the fractional counterpart of De Giorgi's Blow-up Theorem~\cite{maggi12}*{Th.~15.5}, which states that 
if $\mathbf 1_E\in BV_{\rm loc}(\R^N)$ and $x\in\redb E$, the \emph{reduced boundary} of $E$, 
then $\tanset(E,x)=\set{H^+_{\nu_E(x)}(x)}$, where 
\begin{equation*}
H^+_{\nu_E(x)}(x)=\set*{y\in\R^N : (y-x)\cdot\nu_E(x)\ge 0},
\end{equation*}
and $\nu_E\colon\redb E\to\mathbb S^{N-1}$ is the (\emph{measure-theoretic}) \emph{inner unit normal} of $E$. 
In fact, by~\cite{comi-stefani24b}*{Th.~1.12(iii)}, if $\mathbf 1_E\in BV_{\rm loc}(\R^N)$, 
then $\redb E\subset\redb^\alpha E$ and $\nu_E^\alpha=\nu_E$ on~$\redb E$, 
showing that \cref{res:fracblowup} naturally extends the classical setting.
However, the fractional reduced boundary $\redb^\alpha E$ can be substantially larger than $\redb E$, 
even when $E$ is very regular; see~\cite{comi-stefani24b}*{Props.~1.13 and~1.15} for the cases of the half-space and the ball.

\subsection{Main result}

By combining  \cref{res:fracblowup}\ref{item:finiteper} with well-known stability properties 
of the family of tangent sets (see~\cite{leonardi00}*{Props.~2.1 and~2.2} for instance), 
we get that
\begin{equation}
\label{eq:gippo}
\text{$\R^{N-1}\times M\in\tanset(E,0)\setminus\set{\emptyset,\R^N}$ and $P(M)<\infty$}
\implies
H^+_{\mathrm e_N}(0)\in\tanset(E,0).
\end{equation}
This observation motivates the following question: 
does~\eqref{eq:gippo} still hold under the weaker assumption that $\mathbf 1_M\in BV_{\rm loc}(\R)$? 
The purpose of the present note is to provide an affirmative answer.
Here and below, we write $BV_{\rm loc}^\star(\R^N)=BV_{\rm loc}(\R^N)\setminus\set*{0,1}$ for brevity.

\begin{theorem}
\label{res:main}
Given $\mathbf 1_E\in BV^{\alpha,\infty}_{\rm loc}(\R^N)$, for 
$|D^\alpha\mathbf 1_E|$-a.e.\ $x\in\redb^\alpha E$ the following holds.
If $F\in\tanset(E,x)\cap BV_{\rm loc}^\star(\R^N)$, then $F = H_{\nu_E^\alpha(x)}^+(y)$ for some 
$y\in\R^N$. 
In particular,
\begin{equation}
\label{eq:main}
\tanset(E,x)\cap BV_{\rm loc}^\star(\R^N)\ne\emptyset
\implies
H_{\nu_E^\alpha(x)}^+(x)\in\tanset(E,x).
\end{equation}
\end{theorem}

Some comments are in order. 
First, as shown by the examples in~\cite{comi-stefani24b}*{Props.~1.13 and~1.15}, 
the implication~\eqref{eq:main} does not hold when all tangent sets are trivial, so the assumption that a non-trivial blow-up exists is necessary. 
Second, \cref{res:main} yields a rigidity property: for 
$|D^\alpha\mathbf 1_E|$-\emph{almost every} $x\in\redb^\alpha E$, every non-trivial tangent set with locally finite integer perimeter is a half-space oriented by $\nu_E^\alpha(x)$. 
This provides the closest available analogue of De Giorgi's Blow-up Theorem in the fractional setting. 
Third, we do not know whether this rigidity holds at \emph{every} point of $\redb^\alpha E$. 
In particular, extending property~\ref{item:finiteper} in \cref{res:fracblowup} 
to the broader case $\mathbf 1_M\in BV_{\rm loc}(\R)$ remains open.

\cref{res:main} was inspired by~\cite{ambrosio-et-al09}*{Th.~1.2}, where an analogous statement is established in \emph{Carnot groups}.  
The key ingredient is the principle, due to Preiss~\cites{mattila95,preiss87}, that \emph{tangent measures to tangent measures are tangent measures} (see \cref{res:preiss} below).  
This fact makes it possible to reduce the proof to the situation in which a tangent of $E$ has locally finite \emph{integer} perimeter, a case for which the conclusion follows from the classical theory.

A broader motivation for \cref{res:main} comes from the study of sets with locally finite 
distributional fractional perimeter. 
While De Giorgi's theorem describes the infinitesimal structure of sets of finite perimeter, the fractional case lacks an analogous rectifiability theory and even the appropriate infinitesimal models are not fully understood. 
In this direction, \cref{res:main,res:iterated} suggest that certain rigidity features of the classical theory may persist under suitable assumptions on the existence of non-trivial tangent sets.

\section{Preliminaries}

\subsection{Radon measure}

Given an open set $\Omega\subset\R^N$, we let $\mathscr M(\Omega)$ and $\mathscr M_{\rm loc}(\Omega)$ be the spaces of finite and locally finite signed Radon measures on~$\Omega$, respectively.

By the Riesz Representation Theorem, $\mathscr M_{\rm loc}(\Omega)$ is the dual of $C_c(\Omega)$, 
endowed with the topology of local uniform convergence. 
Accordingly, we say that a sequence $(\mu_k)_{k\in\N}\subset\mathscr M_{\rm loc}(\Omega)$ 
\emph{converges to} $\mu\in\mathscr M_{\rm loc}(\Omega)$ 
in the (\emph{local}) \emph{weak$\,^\star$ sense} if 
\begin{equation}
\label{eq:weakstar}
\lim_{k\to\infty}
\int_\Omega \varphi\,\mathrm{d}\mu_k
=
\int_\Omega \varphi\,\mathrm{d}\mu
\quad
\text{for every } 
\varphi\in C_c(\Omega).
\end{equation}
In this case, we write $\mu_k\weakstarto\mu$ 
in $\mathscr M_{\rm loc}(\Omega)$ as $k\to\infty$.

For vector-valued measures, we write 
$\mathscr M(\Omega;\R^m)$ and $\mathscr M_{\rm loc}(\Omega;\R^m)$, with $m\in\N$.
The notion of (local) weak$^\star$ convergence applies to $\R^m$-valued Radon measures by requiring~\eqref{eq:weakstar} to hold componentwise.
In this case, we write $\mu_k\weakstarto\mu$ 
in $\mathscr M_{\rm loc}(\Omega;\R^m)$ as $k\to\infty$.
Accordingly, the \emph{total variation} of $\mu\in\mathscr M_{\rm loc}(\Omega;\R^m)$ on an open set $A\subset\Omega$ is defined as
\begin{equation*}
|\mu|(A)
=\sup\set*{\int_\Omega \varphi\cdot\di\mu : 
\varphi\in C_c(\Omega;\R^m),\ 
\supp\varphi\subset A,\ 
\|\varphi\|_{L^\infty}\le 1}.
\end{equation*}
We recall that $|\mu|\in\mathscr M_{\rm loc}(\Omega)$ for every $\mu\in\mathscr M_{\rm loc}(\Omega;\R^m)$; 
see~\cite{maggi12}*{Lem.~4.17}. 
For a detailed presentation of the theory of Radon measures, we refer to~\cites{ambrosio-et-al00,falconer97,maggi12,mattila95}.

\subsection{Tangent measures}

Given $m\in\N$ and $\mu\in\mathscr M_{\rm loc}(\R^N;\R^m)$, for every $x\in\R^N$ and $r>0$ we define 
$\mu_{x,r}\in\mathscr M_{\rm loc}(\R^N;\R^m)$ by setting 
\begin{equation*}
\mu_{x,r}(A)=\mu(x+rA)
\quad 
\text{for every Borel set } A\subset\R^N.
\end{equation*}
As customary (see~\cite{preiss87}*{2.3(1)}, \cite{mattila95}*{Def.~14.1}, and~\cite{ambrosio-et-al00}*{Sec.~2.7}), 
we say that $\nu\in\tanset(\mu,x)$ if and only if there exist sequences 
$(r_k)_{k\in\N},(c_k)_{k\in\N}\subset(0,\infty)$ such that 
$r_k\to0^+$ and 
\begin{equation*}
c_k\,\mu_{x,r_k}\weakstarto\nu
\quad
\text{in $\mathscr M_{\rm loc}(\R^N;\R^m)$ as } 
k\to\infty.
\end{equation*} 
Moreover (see~\cite{falconer97}*{Sec.~9.1}), given $s\ge0$, we say that $\nu\in\tanset_s(\mu,x)$ if there exists an infinitesimal subsequence $(r_k)_{k\in\N}\subset(0,\infty)$ such that 
\begin{equation*}
r_k^{-s}\,\mu_{x,r_k}\to\nu
\quad
\text{in $\mathscr M_{\rm loc}(\R^N;\R^m)$ as } 
k\to\infty.
\end{equation*} 
By definition, it is clear that $\tanset_s(\mu,x)\subset\tanset(\mu,x)$ 
for every $s\ge0$, $\mu\in\mathscr M_{\rm loc}(\R^N;\R^m)$, and $x\in\R^N$.
For the proof of \cref{res:main}, we will rely on the following crucial result.

\begin{theorem}
\label{res:preiss}
Let $s\ge0$, $m\in\N$ and $\mu\in\mathscr M_{\rm loc}(\R^N;\R^m)$. 
Then, for $|\mu|$-a.e.\ $x\in\R^N$, every $\nu\in\tanset_s(\mu,x)$ satisfies the following properties:
\begin{enumerate}[label=(\roman*),itemsep=1ex,topsep=1ex]
\item 
\label{item:preisstan}
$\nu_{y,r}\in\tanset_s(\mu,x)$ for every $y\in\supp|\nu|$ and $r>0$;
\item
$\tanset_s(\nu,y)\subset\tanset_s(\mu,x)$ for every $y\in\supp|\nu|$.
\end{enumerate}
\end{theorem}

For $m=1$ and $\mu\ge0$, \cref{res:preiss} can be found in~\cite{falconer97}*{Prop.~9.3} and, with $\tanset$ in place of $\tanset_s$, 
in~\cite{preiss87}*{Th.~2.12} and~\cite{mattila95}*{Th.~14.16}. 
Moreover, \cref{res:preiss} corresponds to~\cite{ambrosio-et-al09}*{Th.~6.4}, 
stated on a Carnot group $\mathbb G$ in place of $\R^N$ and under the additional assumption 
that $\mu$ is \emph{asymptotically $s$-regular}; that is,
\begin{equation}
\label{eq:asympreg}
0<
\liminf_{r\to0^+}\frac{|\mu|(B_r(x))}{r^s}
\le 
\limsup_{r\to0^+}\frac{|\mu|(B_r(x))}{r^s}
<\infty
\quad
\text{for $|\mu|$-a.e.\ } x\in\mathbb G.
\end{equation}
Finally, \cref{res:preiss} also corresponds to~\cite{mattila05}*{Prop.~2.15}, 
with $m=1$, $\R^N$ replaced by a homogeneous locally compact metric group $\mathbf G$, 
and $\mu\ge0$ such that 
\begin{equation}
\label{eq:asymdoub}
\limsup_{r\to0^+}
\frac{\mu(B_{2r}(x))}{\mu(B_r(x))}<\infty
\quad
\text{for $\mu$-a.e.\ } x\in\mathbf G.
\end{equation}
The additional assumptions~\eqref{eq:asympreg} and~\eqref{eq:asymdoub} are required to ensure the validity of the Differentiation Theorem, which does not generally hold in an arbitrary metric space.

The proof of \cref{res:preiss} follows almost \emph{verbatim} the argument of~\cite{falconer97}*{Prop.~9.3} 
(which, in turn, is nearly identical to that of~\cite{mattila95}*{Th.~14.16}), 
up to the minor modifications needed to treat the vector-valued case, 
as in~\cite{ambrosio-et-al09}*{Th.~6.4}. 
We therefore omit the details.

\section{Proof of the main result}

\subsection{Properties of tangent sets}

From \cref{res:preiss}, we get the following result.

\begin{corollary}
\label{res:preissset}
Let $\mathbf 1_E\in BV^{\alpha,\infty}_{\rm loc}(\R^N)$. 
Then, for $|D^\alpha\mathbf 1_E|$-a.e.\ $x\in\redb^\alpha E$, every $F\in\tanset(E,x)\setminus\set{\emptyset,\R^N}$ satisfies the following properties:
\begin{enumerate}[label=(\roman*),itemsep=1ex,topsep=1ex]

\item 
$\frac{F-y}{r}\in \tanset(E,x)$ for every $y\in\supp|D^\alpha\mathbf 1_F|$ and $r>0$;

\item\label{item:preisssettanintan}
$\tanset(F,y)\subset\tanset(E,x)$ for every $y\in\supp|D^\alpha\mathbf 1_F|$.
\end{enumerate}
\end{corollary}

In order to prove \cref{res:preissset}, we need the following preliminary result.

\begin{lemma}
\label{res:taniffmeastan}
If $\mathbf 1_E\in BV^{\alpha,\infty}_{\rm loc}(\R^N)$ and $x\in\redb^\alpha E$, then
\begin{equation*}
F\in\tanset(E,x)\setminus\set{\emptyset,\R^N}
\iff
D^\alpha\mathbf 1_F\in\tanset_{N-\alpha}(D^\alpha\mathbf 1_E,x)\setminus\set*{0}. 
\end{equation*}
\end{lemma}

\begin{proof}
If $F\in\tanset(E,x)\setminus\set{\emptyset,\R^N}$, then $\mathbf 1_F\in BV^{\alpha,\infty}_{\rm loc}(\R^N)$ by \cref{res:fracblowup}.
Moreover, we can find an infinitesimal sequence $(r_k)_{k\in\N}\subset(0,\infty)$ such that 
$\frac{E-x}{r_k}\to F$ in $L^1_{\rm loc}(\R^N)$ as $k\to\infty$. 
Therefore, by the scaling properties of the fractional $\alpha$-gradient~\eqref{eq:fracnabla} 
(see also~\cite{comi-stefani19}*{Eq.~(4.8)}) and~\cite{comi-stefani24b}*{Th.~1.6(i)}, we obtain
\begin{equation*}
r_k^{\alpha-N}(D^\alpha\mathbf 1_E)_{x,r_k}
=
D^\alpha\mathbf 1_{\frac{E-x}{r_k}}
\weakstarto
D^\alpha\mathbf 1_F
\quad
\text{in $\mathscr M_{\rm loc}(\R^N;\R^N)$ as $k\to\infty$},
\end{equation*} 
showing that $D^\alpha\mathbf 1_F\in\tanset_{N-\alpha}(D^\alpha\mathbf 1_E,x)$.
In addition, we must have $D^\alpha\mathbf 1_F\ne0$, since otherwise $\mathbf 1_F$ would be constant 
by~\cite{comi-stefani24b}*{Prop.~1.8}, and thus $F\in\set{\emptyset,\R^N}$, a contradiction. 

Conversely, if $D^\alpha\mathbf 1_F\in\tanset_{N-\alpha}(D^\alpha\mathbf 1_E,x)\setminus\set*{0}$, 
then clearly $\mathbf 1_F\in BV^{\alpha,\infty}_{\rm loc}(\R^N)$, and we can find an infinitesimal sequence 
$(r_k)_{k\in\N}\subset(0,\infty)$ such that
\begin{equation*}
r_k^{\alpha-N}(D^\alpha\mathbf 1_E)_{x,r_k}
\weakstarto
D^\alpha\mathbf 1_F
\quad
\text{in $\mathscr M_{\rm loc}(\R^N;\R^N)$ as $k\to\infty$}.
\end{equation*}
Thus, letting $E_k=(E-x)/r_k$ for every $k\in\N$, the sequence $(E_k)_{k\in\N}$ satisfies 
$\mathbf 1_{E_k}\in BV^{\alpha,\infty}_{\rm loc}(\R^N)$ for every $k\in\N$, and
\begin{equation*}
D^\alpha\mathbf 1_{E_k}
=
r_k^{\alpha-N}(D^\alpha\mathbf 1_E)_{x,r_k}
\weakstarto
D^\alpha\mathbf 1_F
\quad
\text{in $\mathscr M_{\rm loc}(\R^N;\R^N)$ as $k\to\infty$}.
\end{equation*}
Since $x\in\redb^\alpha E$, possibly passing to a subsequence (which we do not relabel), 
by \cref{res:fracblowup} there exists $G\in\tanset(E,x)$ such that 
$\mathbf 1_G\in BV^{\alpha,\infty}_{\rm loc}(\R^N)$ and $E_k\to G$ in $L^1_{\rm loc}(\R^N)$ as $k\to\infty$. 
Again by~\cite{comi-stefani24b}*{Th.~1.6(i)}, possibly passing to a further subsequence (which we do not relabel), we also have that 
\begin{equation*}
D^\alpha\mathbf 1_{E_k}\weakstarto D^\alpha\mathbf 1_G
\quad
\text{in $\mathscr M_{\rm loc}(\R^N;\R^N)$ as $k\to\infty$},
\end{equation*} 
from which it follows that $D^\alpha\mathbf 1_G=D^\alpha\mathbf 1_F$.
By~\cite{comi-stefani24b}*{Prop.~1.8}, this implies that $\mathbf 1_G-\mathbf 1_F$ is constant, 
so that either $F=G$ or $F=\R^N\setminus G$.
The latter possibility is ruled out, since otherwise 
$D^\alpha\mathbf 1_F=-D^\alpha\mathbf 1_G$ and thus $D^\alpha\mathbf 1_F=0$, a contradiction. 
Hence $F=G$, and therefore $F\in\tanset(E,x)$.
Moreover, $F\notin\set{\emptyset,\R^N}$, since  $D^\alpha\mathbf 1_F\ne0$.
\end{proof}

\begin{proof}[Proof of \cref{res:preissset}]
If $F\in\tanset(E,x)\setminus\set{\emptyset,\R^N}$, then by \cref{res:taniffmeastan} we deduce that 
$D^\alpha\mathbf 1_F\in\tanset_{N-\alpha}(D^\alpha\mathbf 1_E,x)\setminus\set{0}$.
Therefore, by \cref{res:preiss}\ref{item:preisstan}, we get that 
$(D^\alpha\mathbf 1_F)_{y,r}\in\tanset_{N-\alpha}(D^\alpha\mathbf 1_E,x)$ 
for every $y\in\supp|D^\alpha\mathbf 1_F|$ and every $r>0$.
Since $(D^\alpha\mathbf 1_F)_{y,r}=D^\alpha\mathbf 1_{(F-y)/r}$ by the scaling properties of~\eqref{eq:fracnabla}, 
again by \cref{res:taniffmeastan} this implies that $(F-y)/r\in\tanset(E,x)$ 
for every $y\in\supp|D^\alpha\mathbf 1_F|$ and $r>0$. 
Finally, since $\tanset(E,x)$ is closed with respect to convergence in $L^1_{\rm loc}(\R^N)$ 
(see~\cite{leonardi00}*{Prop.~2.2}), we also obtain that 
$\tanset(F,y)\subset\tanset(E,x)$ for every $y\in\supp|D^\alpha\mathbf 1_F|$, 
concluding the proof.    
\end{proof}

\subsection{Iterated tangent sets}

For the proof of \cref{res:main}, we rely on the notion of \emph{iterated} tangent sets.
Precisely, given $\mathbf 1_E\in BV^{\alpha,\infty}_{\rm loc}(\R^N)$ and $x\in\redb^\alpha E$, 
we define $\tanset^1(E,x)=\tanset(E,x)$ and 
\begin{equation}
\label{eq:defktan}
\tanset^{k+1}(E,x)
=
\bigcup
\set*{\tanset(F) : F\in\tanset^k(E,x)}
\end{equation}
for all $k\in\N$, where
\begin{equation*}
\tanset(E)
=
\bigcup\set*{\tanset(E,x):x\in\redb^\alpha E}.
\end{equation*}
The proof of \cref{res:main} is based on the following result, 
which may be of independent interest (and rephrases~\cite{ambrosio-et-al09}*{Th.~6.1} in the present setting).  

\begin{theorem}
\label{res:iterated}
If $\mathbf 1_E\in BV^{\alpha,\infty}_{\rm loc}(\R^N)$, then for $|D^\alpha\mathbf 1_E|$-a.e.\ $x\in\redb^\alpha E$ it holds that
\begin{equation*}
\bigcup_{k=2}^\infty
\tanset^k(E,x)
\subset
\tanset(E,x).
\end{equation*}
\end{theorem}

\begin{proof}
Let $x\in\redb^\alpha E$ be such that \cref{res:preissset}\ref{item:preisssettanintan} holds and let $F\in\tanset(E,x)$.
If $F\in\set{\emptyset,\R^N}$, then $\redb^\alpha F=\emptyset$ and there is nothing to prove.
If instead $F\notin\set{\emptyset,\R^N}$, then for every $y\in\redb^\alpha F$ we have that $\tanset(F,y)\subset\tanset(E,x)$.
By the definition in~\eqref{eq:defktan}, we infer that $\tanset^2(E,x)\subset\tanset^1(E,x)$,
and the conclusion follows by iteration. 
\end{proof}

\subsection{Proof of \texorpdfstring{\cref{res:main}}{Theorem 1.2}}
We can now prove our main result.

\begin{proof}[Proof of \cref{res:main}]
Let $x\in\redb^\alpha E$ be such that \cref{res:iterated} holds.
By assumption, there exists $F\in\tanset(E,x)\setminus\set{\emptyset,\R^N}$ such that $\mathbf 1_F\in BV_{\rm loc}(\R^N)$.
Hence, we have that $\tanset(F,y)=\set{H_{\nu_F(y)}^+(y)}$ for every $y\in\redb F$.
Since $\redb F\subset\redb^\alpha F$ by~\cite{comi-stefani24b}*{Th.~1.12(iii)}, 
from the definition in~\eqref{eq:defktan} we deduce that $H_{\nu_F(y)}^+(y)\in\tanset^2(E,x)$ for every $y\in\redb F$. 
This, in turn, by \cref{res:iterated}, implies that $H_{\nu_F(y)}^+(y)\in\tanset(E,x)$ for every $y\in\redb F$.
Combining~\cite{comi-stefani24b}*{Prop.~1.13} with \cref{res:fracblowup}, we then obtain that, for every $y\in\redb F$,
\begin{equation*}
\nu_F(y)=\nu^\alpha_{H^+_{\nu_F(y)}(y)}(z)=\nu^\alpha_E(x)
\quad
\text{for a.e.}\ z\in\R^N.
\end{equation*}
As a consequence, $\nu_F(y)=\nu^\alpha_E(x)$ for every $y\in\redb F$, 
which implies that $F=H^+_{\nu^\alpha_E(x)}(x_0)$ for some $x_0\in\R^N$ 
by~\cite{maggi12}*{Prop.~15.15}, and~\eqref{eq:main} follows by~\cite{leonardi00}*{Props.~2.1 and~2.2}. 
\end{proof} 


\end{document}